\documentclass{amsart}
\usepackage{tikzsymbols}
\usepackage{enumitem}
\usepackage{amsmath, amssymb, amsthm}
\usepackage{tikz}
\usetikzlibrary {fit}
\usetikzlibrary {chains}
\usetikzlibrary{calc}

\DeclareMathOperator{\E}{E}
\newcommand{\bbN}{\mathbb N}
\newcommand{\bbC}{\mathbb C}
\newcommand{\bbR}{\mathbb R}
\newcommand{\bbM}{\mathbb M}
\newcommand{\HS}{\textrm{HS}}
\newcommand{\OP}{\textrm{OP}}
\newcommand{\cS}{\mathcal S}
\newcommand{\cU}{\mathcal U}
\newcommand{\cV}{\mathcal V}

\newtheorem{theorem}{Theorem}[section]
\newtheorem*{theorem*}{Theorem}

\newtheorem{question}[theorem]{Question}
\newtheorem*{proposition*}{Proposition}
\newtheorem{lemma}[theorem]{Lemma}
\newtheorem*{lemma*}{Lemma}

\newtheorem*{corollary*}{Corollary}
\newcommand{\mycircle}[3]{%
	\begin{tikzpicture}
		\def\radius{#1}
		\def\angles{9,80,80+#3,230-#3,230}
		\def\gap{2.5}
		\draw [color=blue]({\radius*cos(9+\gap)}, {\radius*sin(9+\gap)}) arc (9+\gap:80-\gap:\radius);
		\draw [color=red] ({\radius*cos(80+\gap)}, {\radius*sin(80+\gap)}) arc (80+\gap:80+#3-\gap:\radius);
		\draw [color=blue]({\radius*cos(80+#3+\gap)}, {\radius*sin(80+#3+\gap)}) arc (80+#3+\gap:230-\gap:\radius)  ;
		\draw [color=red] ({\radius*cos(230-#3-#3+\gap)}, {\radius*sin(230-#3-#3+\gap)}) arc (230-#3-#3+\gap:360+9-\gap:\radius);
		\draw [color=blue]({\radius*cos(230+\gap)}, {\radius*sin(230+\gap)}) arc (230+\gap:360+9-\gap:\radius);
		\node at ({\radius*cos(9)}, {\radius*sin(9)}) {#2$_1$};
		\node at ({\radius*cos(80)}, {\radius*sin(80)}) {#2$_3$};
		\node at ({\radius*cos(80+#3)}, {\radius*sin(80+#3)}) {#2$_2$};
		\node at ({\radius*cos(230-#3-#3)}, {\radius*sin(230-#3-#3)}) {#2$_5$};
		\node at ({\radius*cos(230)}, {\radius*sin(230)}) {#2$_4$};
\end{tikzpicture}}

\title{Games in the matrix: Director's cut}
\author{Ilijas Farah}
\date{\today}
\address{Department of Mathematics and Statistics\\
	York University\\
	4700 Keele Street\\
	North York, Ontario\\ Canada, M3J 1P3\\
	and 
	Ma\-te\-ma\-ti\-\v cki Institut SANU\\
	Kneza Mihaila 36\\
	11\,000 Beograd, p.p. 367\\
	Serbia}
\email{ifarah@yorku.ca}
\urladdr{https://ifarah.mathstats.yorku.ca}
\thanks{Partially supported by NSERC}
\thanks{This is an expanded version of \cite{farah2026games}. While movie directors routinely complain that the studio butchered their masterpiece, I am indebted to Victor Bloch, whose adaptation of this article—including the addition of exercises—for \emph{Snapshots of modern mathematics from Oberwolfach} greatly improved the original. The present version includes further details as well as the exercises kindly provided by Victor.}
\begin{document}
	
		\begin{abstract}
		As $n\to \infty$, do the algebras of $n\times n$ complex matrices look alike? Nobody knows, but let's play a game and see how far we get! 
	\end{abstract}
	\maketitle

	\section{Games}
	
	Let's start with the simplest interesting mathematical structures. 
	A \emph{simple graph} is a set of vertices some of which are connected by an `edge' (the technical term is `adjacent'). What makes a graph simple is that only distinct nodes can be adjacent and that we do not allow multiple edges between any two vertices. We write $G=(V,E)$, where~$V$ is the set of vertices and $E$ is the set of connections, called edges. If $V$ has $m$ elements, then there are exactly $ m\choose  2$ ways to choose the edges, because for every pair of distinct vertices we have two choices, whether to connect them with an edge or not. 
	To compare two graphs, we will play a game. 
	Given graphs $G_1=(V_1,E_1)$ and $G_2=(V_2,E_2)$ and $n\geq 1$, the game $\Gamma(G_1,G_2,n)$ between two players, Challenger and Duplicator is played in~$n$ innings and has the following rules.   In each inning Challenger chooses a vertex in either one of the graphs. Duplicator responds  by choosing a vertex the other graph. Thus in the $j$-th inning a vertex $a_j$ in $V_1$ and a vertex $b_j$ in~$V_2$ are chosen. We add the (unnecessary)  requirement that vertices cannot be repeated; the first player who repeats  a vertex loses the game.  After $n$ innings,  Duplicator wins if all $i$ and $j$ satisfy $a_i \E_1 a_j$ if and only if $b_i \E_2 b_j$ (another way to say this is that the map $a_i\mapsto b_i$, for $i\leq n$, is an \emph{isomorphism between the induced subgraphs}); otherwise Challenger wins. 
	
	Let's get some easy observations out of the way. 
	If there is an isomorphism~$\Phi$ between the graphs, then Duplicator can win the game by applying $\Phi$ or its inverse to the vertices provided by the Challenger. If both graphs have $n$ vertices,  then a game in which Duplicator wins produces an isomorphism between the two graphs.
	
	Each game considered in this note is a finite game of perfect information, meaning that each player has complete knowledge of the history of the game up to their current turn.  For example, both chess and go are finite games of perfect information, but in each of these games draw is a possible outcome. Since draw is not a possible outcome in any of the games considered in this note, being a finite game of perfect information implies that one of the players has a winning strategy.\footnote{The question of the existence of a strategy for one of the players (a property called determinacy) for infinite games is the first in a long list of fascinating topics related to our subject that we regrettably do not have time to discuss.} 
	A \emph{winning strategy} in $\Gamma(G_1,G_2,n)$ for Duplicator is a function from positions in the game in which it is Duplicator's turn whose values are valid moves in the game, and such that by following this strategy Duplicator wins the game regardless of how Challenger plays. A winning strategy for Challenger is defined analogously. If each of the players had a winning strategy, then pitting these strategies against each other would result in a game in which both players win; contradiction. An example of a winning strategy is provided in the proof of Theorem~\ref{T.EF-graph} below.
	In chess and go there are three possible outcomes: (i--ii) One  of the players has a winning strategy,  or (iii) each of the players has a strategy that assures that they do not lose the game.

	%
	%
	%
	%
	%
	
	We move on to a (somewhat) nontrivial example of a winning strategy for Duplicator. For $m\geq 3$ let~$C_m$ be the \emph{cycle graph}, whose vertices are $\{1,\dots, m\}$ and two vertices~$i,j$ are adjacent if and only if $|i-j|=1$ or $\{i,j\}=\{1,m\}$. 
	
	\begin{theorem}\label{T.EF-graph}  If $\min (m,k)>2^{n+1}$, then Duplicator wins  $\Gamma(C_m,C_k,n)$. 
	\end{theorem}
	
	\begin{proof}The \emph{distance}, $d(x,y)$ between two vertices $x,y$ in a graph is the smallest $l$ such that there is a path from $x$ to $y$ of length $l$ in which each step consists of moving to an adjacent vertex.\footnote{If there is no such $l$ then we  set $d(x,y)=\infty$.}
		For $j\geq 1$ we say that two vertices $x,y$ in $C_m$ are $j$-far if $d(x,y)>2^{j}$, and that a path is $j$-long if its length is $>2^{j}$. Note that vertices are 0-near if and only if they are adjacent. 
		We can now describe the strategy for the Duplicator. The positions in this game are so similar that there is no harm in pretending that the Challenger always chooses vertices from the first graph.  To increase the stakes, we switch to the  first-person narrative. In the first inning, we can play any vertex $b_1$. In the second move, we have vertices $a_1$ and $a_2$ in one of the graphs and $b_1$ in the second. If one of the paths is $n$-short (meaning, not $n$-long), then choose $b_2$ so that $d(b_1,b_2)=d(a_1,a_2)$. The salient point is that, because
		\begin{figure}[h]
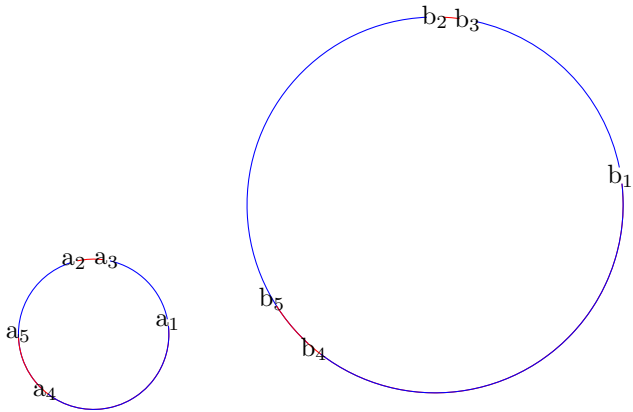

			\mycircle {1} a {25}\qquad 
			\mycircle {2.5} b {10}
			\caption{The position after five innings. Since $m$ and $l$ are very large, we zoomed out so much that the vertices blended into two continuous circles. The corresponding short (red) arcs are of the same length, while the long (blue) arcs are simply sufficiently long. No matter what vertex the Challenger plays, we can respond  by choosing a vertex in the corresponding arc that is either at the right distance or sufficiently far from the corresponding ends.}
		\end{figure}
		$\min(m,k)>2^{n+1}$,  the longer path between $a_1$ and $a_2$ is $n$-long, and so is the longer path between $b_1$ and $b_2$.  If $a_1$ and $a_2$ are $n$-far, then take $b_2$ to be any vertex $n$-far from $b_1$; this is possible because $\min(m,k)>2^n$. At this stage we have vertices $a_1$ and $a_2$ `splitting' one cycle into two arcs and $b_1$ and $b_2$ `splitting' the other cycle into two arcs. Either all of the arcs are sufficiently long, or the short arcs are of the same length. 
		
		Our strategy in the $k+1$-st inning for $k\geq 2$ hinges on the following inductive hypothesis  holding after the $k$-th inning. 
		
		\begin{itemize}
			\item [(IH$_k$)] The vertices $a_1,\dots, a_k$ and $b_1,\dots, b_k$ are arranged on their cycles so that for all $i$ and $j$ either $d(a_i,a_j)=d(b_i,b_j)$ or $a_i$ and $a_j$ are $k+1$-far and $b_i$ and $b_j$ are $k+1$-far. 
		\end{itemize}
		After the $k$-th inning for $k\geq 2$, Challenger adds a vertex $a_{k+1}$ to one of the arcs, between vertices $a_i$ and $a_j$ for some $i,j$. If it is added to a short arc, then we can choose $b_3$ so that sides of the `triangle'\footnote{These triangles are rather flat, but no worries.} $b_i,b_j,b_k$ are equal to the corresponding sides of the `triangle' with the vertices $a_i,a_j,a_k$.
		If $a_k$ is added to a long arc, but it is $n-1$-near $a_i$, then we play $b_k$  on the arc $b_i,b_j$ so that $d(b_k,b_i)=d(a_k,a_i)$. Otherwise, play~$b_k$ that is $k$-far from both $b_i$ and $b_j$.
		After this move, the inductive hypothesis is satisfied. This shows that we can hold on for at least $n$ moves. When the game is over, $b_i$ and $b_j$ are $0$-near if and only if $a_i$ and $a_j$ are 0-near. Since `0-near' is the same as `adjacent',  we win. 
	\end{proof}
	
	This tells us that sufficiently long cycles are similar.  What else is new, you may wonder? What if we choose graphs at random, by fixing the vertex set and then for each pair of vertices flipping a coin to decide whether they are adjacent?
	If for all $j$,~$G_j$ is a randomly chosen graph with $j$ vertices, then for every $n$ there is $f(n)$ such that if $\min(m,k)\geq f(n)$ then Duplicator wins $\Gamma(G_m,G_k,n)$. This follows from the so-called 0-1 law for random graphs (\cite{shelah1988zero}) and an observation that requires a syntactical definition.\footnote{A mathematical logician is a mathematician who takes  the syntax seriously.} 
	. 
	
	A first-order statement about graphs is expressed by quantification over the vertices, and it involves the adjacency  and equality relations. We will consider only statements that  have no free variables; they are called `sentences' (as opposed to `formulas').\footnote{A second-order statement is one in which quantification over arbitrary subsets of the structure is allowed. It is considerably more expressive than the first-order logic. The latter lies in the `Goldilocks zone' in terms of expressiveness, and in the final section we will see some of the evidence for this.}
	and it involves the adjacency relation $\E$ and the equality relation $=$. For this purpose we have to introduce some mathematical (or rather logical) notation: First of all, there is the \emph{universal quantifier} $\forall$ which reads as ``for all'' and the \emph{existential quantifier}~$\exists$ which reads as ``there exists''. Since we are interested in first-order statements about graphs quantified over the vertices, $\forall x$ reads as ``for all vertices $x$'' and similarly $\exists x$ reads as ``there exists a vertex $x$''. Furthermore, we need the logical connectives $\land$ (``and''), $\lor$ (``or''), $\lnot$ (``not''), and $\rightarrow$ (``implies''). For example,
	\[ 
	\varphi_1: (\exists x)(\forall y)y\neq x\to y\E x
	\] 
	reads as ``there is a vertex $x$ such that for all vertices $y$, $y\neq x$ implies that $y$ and $x$ are connected by an edge'' and hence asserts that some vertex is adjacent to all other vertices. Now try it yourself! Can you translate the following statement?
	\[
	\varphi_2: (\forall x) (\forall y) x\neq y\to (\exists z) z\E x\wedge \lnot (z\E y) \ \footnotemark
	\]
	But what does this  have to do with games? Well, everything. But before we go into the details, have a look at the two graphs in Appendix~\ref{game1}. Play the game on the two graphs in three innings. Can you find a winning strategy for either Challenger or Duplicator?  \footnotetext{Solution: ``For all vertices $x$ and for all vertices $y$, $x \neq y$ implies that there exists a vertex $z$ such that an edge connects $z$ and $x$ and no edge connects $z$ and $y$.'' This means that for any two distinct vertices $x,y$ some vertex is connected to $x$ but not to~$y$.}
	
	Are you back? Here is the solution. For any two graphs $G_1$ and $G_2$, if~$G_1$ satisfies~$\varphi_2$ and~$G_2$ does not (check that this is true for the two graphs in Appendix~\ref{game1}), then Challenger wins $\Gamma(G_1,G_2,3)$ as follows. In the first two innings, Challenger chooses $a_1$ and $a_2$ in $G_2$ so that no $z$ satisfies $z\E a_1$ and $\lnot (z\E a_2)$. No matter how $b_1$ and $b_2$ are chosen, in the third inning we can choose $b_3$ so that $b_3\E b_1$ and $\lnot (b_3\E b_2)$, and Duplicator is doomed. 
	
	Now have a look at the graphs in Appendix~\ref{game2}. Play the game on the two graphs with varying numbers of innings. How many innings are necessary for a winning strategy for Challenger to exist?\footnote{The answer is 4.} 
	
	A fun little calculation shows that as the number of vertices in a sufficiently random graph converges to $\infty$, the probability that $\psi$ holds in it converges to 0 and the probability that $\varphi$ holds in it converges  to $1$. The 0-1 law asserts that every first-order statement behaves in one of these two ways: in a sufficiently large and sufficiently random graph,  it is either almost certainly true or almost certainly false.

	But what does this  have to do with games? Well, everything. If $G_1$ satisfies~$\varphi$ and~$G_2$ does not, then the Challenger wins $\Gamma(G_1,G_2,3)$, as follows.   We switch perspective and identify with the winning side. In the first two innings, we choose $a_1$ and $a_2$ in $G_2$ so that no $z$ satisfies $z\E a_1$ and $\lnot (z\E a_2)$. No matter what the choice of $b_1$ and $b_2$ is, in the third inning we can choose $b_3$ so that $b_3\E b_1$ and $\lnot (b_3\E b_2)$, and the Duplicator is doomed. 
	
	For every first-order statement $\psi$ there is $n$ large enough so that the Duplicator loses $\Gamma(G_1,G_2,n)$ if $G_1$ and $G_2$ `disagree' on the truth of $\psi$. The least such $n$ is the \emph{quantifier depth} of $\psi$;  let's just say that the name is well-chosen instead of giving a proper definition.
	Conversely, if $G_1$ and $G_2$ agree on the truth of all first-order statements of quantifier depth $n$, then Duplicator wins $\Gamma(G_1,G_2,n)$. 
	
	These are the Ehrenfeucht--Fra\" isse games, affectionately known as the EF-games (see \cite[\S 3.3]{Hodg:Model}).  
	Their variations can be used to compare mathematical structures of all sorts:  linear orderings, groups, rings,\dots{}  In each case, the EF-games provide efficient means for verifying whether two given structures satisfy the same first-order statements.    Before moving on, we state another fun fact.  It applies to arbitrary structures but for simplicity we stick to graphs.
	
	\begin{theorem}\label{T.graphs-convergence}
		If $G_m$, for $m\in \bbN$, is any sequence of graphs, then there is  a subsequence $G_{m(i)}$, for $i\in \bbN$, such that for every $n$ there is $k(n)$ such that if $\min(m(i),m(j))>k(n)$ then Duplicator wins $\Gamma(G_{m(i)}, G_{m(j)}, n)$. 
	\end{theorem}
	
	Thus every sequence of graphs has a subsequence such that any two graphs in it that are sufficiently `further down the road' cannot be distinguished by a short EF-game. To prove Theorem~\ref{T.graphs-convergence}, one observes that the set of first-order sentences in the language of graph theory can be enumerated as $\varphi_j$, for $j\in \bbN$. 
	It then suffices to assure that our sequence satisfies the `0-1 law', that all but finitely many $G_{m(i)}$ `agree' on the truth of  every $\varphi_j$. This is achieved by a diagonalisation argument.

	\section{Matrices}
	
	During a legendary 1925 vacation in Heligoland,   Werner Heisenberg noticed that because of the noncommutativity of matrix multiplication, matrix algebras provide natural setting for quantum mechanics. A few years later, John von Neumann pointed out that the infinite-dimensional analog of $\bbM_m$ is more suitable for this purpose, and suggested algebras of operators known as II$_1$ factors as a model. We will see a II$_1$ factor (or perhaps many II$_1$ factors---that is the question  that we are aiming at) in the final section.  Let's first take a closer look at what the matrix algebras are all about. 
	
	Consider the (orthogonal) projection $p$ to the line $x=y$, or the   rotation $\rho_\alpha$ of the plane by an angle $\alpha$ around the origin (Fig. \eqref{Fig.RotationAndProjection}). 
	\begin{figure}[h]
		\begin{tikzpicture}[domain=-0.2:3.2]
			\def\offset{0.2}
			\draw[->] (-0.2,0) -- (3.2,0) node[right] {$x$};
			\draw[->] (0,-0.2) -- (0,3.2) node[above] {$y$};
			\draw   [color=red]  plot (\x,\x)             node[right] {$x=y$};
			\node at (1,3) {$\bullet$};
			\node at (1+\offset,3+\offset) {$\mu$};
			\node at (2,0) {$\bullet$};
			\node at (2+\offset,0+\offset) {$\nu$};
			\node at (1,1) {$\bullet$}; 
			\node at (1+\offset, 1+\offset) {$p(\nu)$};
			\node at (2,2) {$\bullet$}; 
			\node at (2+\offset, 2+\offset) {$p(\mu)$};
		\end{tikzpicture}
		\begin{tikzpicture}[domain=-0.2:3.2]
			\def\offset{0.2}
			\def\radius{4}
			\draw[->] (-2.2,0) -- (2.5,0) node[right] {$x$};
			\draw[->] (0,-0.2) -- (0,3.2) node[above] {$y$};
			\node at (1/2, 2) {$\bullet$};
			\node at (1/2+\offset, 2+\offset) {$\xi$}; 
			\node at (1, 0) {$\bullet$};
			\node at (1+\offset, 0+\offset) {$\eta$}; 
			\node at (0,1) {$\bullet$};
			\node at (0+\offset, 1+\offset) {$\rho_{\pi/2}(\eta)$}; 
			\node at (-2,1/2) {$\bullet$}; 
			\node at (-2-\offset,1/2-\offset) {$\rho_{\pi/2}(\xi)$}; 	
			\draw (0,1) (1,0) arc (0:86:1) [->]; 
			\draw (-2,1/2) (1/2,2)  arc (76:163:2.06) [->]; 
		\end{tikzpicture}
		\caption{\label{Fig.RotationAndProjection}The projection to the line~$x=y$ and the rotation by $\pi/2$ around the origin as applied to some points randomly chosen in the plane.}
	\end{figure}
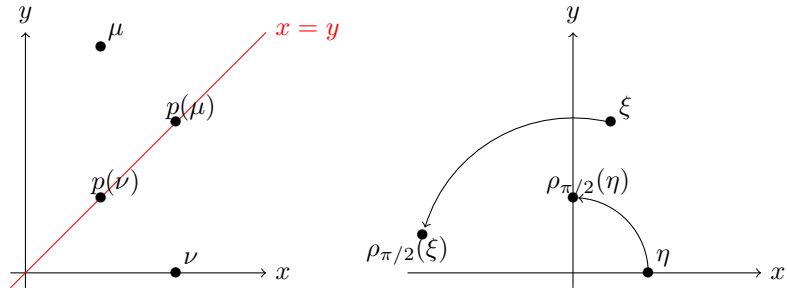
	These transformations are implemented by matrices, via
	\begin{align*}
		\qquad
		p\left(\begin{pmatrix} a\\ b\end{pmatrix}\right)&=
		\begin{pmatrix} 
			\frac a2  + \frac b2 \\
			\frac a2  + \frac b2 	
		\end{pmatrix}
		=
		\begin{pmatrix} 
			\frac 12 & \frac 12 \\
			\frac 12 & \frac 12
		\end{pmatrix}
		\begin{pmatrix} a\\ b\end{pmatrix},\\ 
		\rho_\alpha\left( \begin{pmatrix} a\\ b\end{pmatrix}\right)&=
		\begin{pmatrix} 
			\cos(\alpha)a  -\sin(\alpha)b \\
			\sin(\alpha)a + \cos(\alpha)b
		\end{pmatrix}
		=\begin{pmatrix} 
			\cos(\alpha) & -\sin(\alpha) \\
			\sin(\alpha) & \cos(\alpha)
		\end{pmatrix}
		\begin{pmatrix} a\\ b\end{pmatrix}.
	\end{align*}
	In case you haven't seen it yet, the product of a $2\times 2$ matrix and a 2-vector is defined as follows: the first entry of the image of a vector 	$\begin{pmatrix} a\\ b\end{pmatrix}$ is the dot product  of the first row of the matrix with the vector, and its second entry is the dot product  of the first row of the matrix with the vector. This works in the higher-dimensional context, for example when $n=3$:
	\[
	\begin{pmatrix}
		\color{red} {c_{11}} & \color{red}{c_{12}} & \color{red}{c_{13}}\\
		\color{blue}{c_{21}} & \color{blue}{c_{22}} & \color{blue}{c_{23}}\\
		\color{cyan}{c_{31}} & \color{cyan}{c_{32}} & \color{cyan}{c_{33}}
	\end{pmatrix}
	\begin{pmatrix}
		a_1 \\ a_2 \\ a_3
	\end{pmatrix}
	=
	\begin{pmatrix}
		\color{red}{\sum_{i=1}^3 c_{1i}a_i} \\
		\color{blue}{\sum_{i=1}^3 c_{2i}a_i} \\
		\color{cyan}{\sum_{i=1}^3 c_{3i}a_i }.
	\end{pmatrix}.
	\]
	To multiply two $m\times m$ matrices, $a$ and $b$, one applies this formula to each column of $b$ and produces the corresponding column of the result,~$ab$. 
	
	Fix $m\geq 2$ and let $\bbC^m$ be the space of $m$-tuples (called $m$-vectors) of complex numbers.\footnote{Why not the real numbers? It is a long story, but ultimately because we like every polynomial to have a root and we prefer an isometry such as $\rho_{\pi/2}$ to be diagonalisable. However, all that I am about to say makes perfect sense if $\bbC$ is replaced by $\bbR$, and a reader who feels more comfortable believing that $x^2\geq 0$ for all $x$ is encouraged to do so, ignore the adjoint operation defined below, and skip the final section.}  This is a \emph{vector space} over $\bbC$; what this means is that we can add vectors pointwise and multiply them by a complex number (called scalar).  Some $T\colon \bbC^m\to \bbC^m$ is a \emph{linear transformation} if for all vectors $v, w$ and all scalars $x,y$ it satisfies 
	\[
	T(xv+ yw)=x T(v)+yT(w). 
	\]
	\begin{lemma}
		Every  linear transformation of $\bbC^m$ is implemented via multiplication by an $m\times m$ matrix, and multiplication by a fixed $m\times m$ matrix is a linear transformation of $\bbC^m$. 
		Multiplication of matrices corresponds to composition of linear transformations. \end{lemma}
	
	Thus studying linear transformations of the $m$-dimensional space is the same as  studying the algebra of $m\times m$ matrices; we opt for the latter. 
	We will also need a third operation, the adjoint $a^*$: to compute it, transpose the matrix and take complex conjugates of its entries. Thus $((a_{ij})_{i,j\leq m})^*=(\overline a_{j,i})_{i,j\leq m}$. 
	The adjoint of the matrix that corresponds to a rotation $\rho_{\alpha}$ is the matrix that implements the rotation $\rho_{-\alpha}$, i.e., its multiplicative inverse. The matrices whose adjoint is their multiplicative inverse are called \emph{unitary}. A matrix that corresponds to an (orthogonal)  projection is its own adjoint; such matrices are called \emph{self-adjoint} or \emph{Hermitian}. 
	
	By $\bbM_m$ we denote the algebra of all complex $m\times m$ matrices $a=(a_{ij})_{i,j\leq m}$, equipped with multiplication, addition, adjoint, and multiplication by scalars. (Multiplying a matrix by a scalar appropriately stretches---or shrinks--a given transformation.) By~$1_m$ we denote the \emph{unit matrix} that has $1$'s down the main diagonal and 0's elsewhere; in other words, $a_{ij}=1$ if $i=j$ and $a_{ij}=0$ otherwise. It has the property that $1_m a=a 1_m=a$ for all $a$ in~$\bbM_m$.

	\subsection{Let the games begin}
	For $l,m$, and $n$ one could  consider the following game, $\Gamma_{\textrm{ts}}(\bbM_l,\bbM_m,n)$. 
	Our two players choose elements from $\bbM_l$ and $\bbM_m$  as they did for graphs, and Duplicator wins if the chosen matrices $a_1,\dots, a_n$ and $b_1,\dots , b_n$ `look the same'. This means that for all $i,j,k$  and all   complex numbers $y,z$  we have 
	$a_i a_j=a_k$ if and only if $b_i b_j =b_k$, $ya_i+za_j=a_k$ if and only if $yb_i+zb_j = b_k$, and $a_i^*=a_j$ if and only if $b_i^*=b_j$. 
	
	One could indeed consider this game, but we will not (the subscript `ts' stands for `too strict'). Strict equality is fine for graphs and other discrete structures, but when in the matrix there is no need for exactness.\footnote{Please don't read much into this sentence.} Let's make this precise.

	The length of an $m$-vector $\xi$ with coordinates $x_1,\dots, x_m$ is defined as  
	\[
	\textstyle\|\xi\|=\sqrt {\sum_{i=1}^m |x_i|^2}. 
	\]
	This is justified by the Pythagorean theorem if $m=2$, supplemented with an induction argument for larger values of $m$. 
	There are at least two common ways to define the length (called `norm') of a matrix. 
	First, every $a=(a_{ij})_{i,j\leq m}$ in~$\bbM_m$ is an $m^2$-vector, and its \emph{normalized Hilbert--Schmidt} norm is 
	\[
	\textstyle \|a\|_\HS= \sqrt {\frac 1m\sum_{i,j=1}^m |a_{ij}|^2}. 
	\]
	It is `normalized' because of the factor $\frac 1m$, assuring that $\|1_m\|_\HS=1$. 
	
	The \emph{operator norm} may be slightly less intuitive and it is notoriously difficult to compute for an arbitrary matrix, but it is in some sense even more natural than~$\|\cdot\|_\HS$. It is defined as 
	\[
	\|a\|_\OP=\sup\{\|a\xi\|: \|\xi\|=1\}. 
	\]
	This is the `maximal stretching coefficient': $\|a\|_\OP=r$ means that $\|a\xi\|=r\|\xi\|$ for some\footnote{Do you see why is the sup attained?} nonzero $\xi$ and $\|a\xi\|\leq r\|\xi\|$ for all $\xi$. 
	One has $\|1_m\|_\OP=\|1_m\|_\HS=1$ and $\|\cdot\|_\HS\leq \|\cdot\|_\OP$.  On the other hand, any $m\times m$ matrix $a$ that has $m$ in one entry and 0 in all other entries satisfies $\|a\|_\HS=1$ and $\|a\|_\OP=m$.  
	
	\subsection{Let the correct games begin}
	For $l,m, n$, and $\varepsilon>0$,   we will consider the game $\Gamma_\HS(\bbM_l,\bbM_m,n,\varepsilon)$ which is played like  
	$\Gamma_{\textrm{ts}}(\bbM_l,\bbM_m,n)$, with two differences: 
	\begin{enumerate}[label = (P\arabic*), itemindent=10pt]
		\item \label{P1}Players can choose only matrices $x$ that satisfy $\|x\|_\OP\leq 1$.\footnote{The subscript $\OP$ is not a typo. The correct ambient for our purposes is the operator unit ball of $\bbM_m$, regardless of which one of the norms we consider..} 
		\item \label{P2} Duplicator wins  if and only if for all $i,j,k$  and all scalars $y,z$ such that $\max(|y|,|z|)\leq 1$,  each of $|\|a_i\|_\HS-\|b_i\|_\HS|$, 
		$|\|a_i a_j-a_k\|_\HS-\|b_i b_j -b_k\|_\HS|$, $|\|y a_i+za_j-a_k\|_\HS-\|y b_i+zb_j -b_k\|_\HS|$, and $|\|a_i^*-a_j\|_\HS-\|b_i^*-b_j\|_\HS|$ is no greater than~$\varepsilon$.
	\end{enumerate}
	The game $\Gamma_\OP(\bbM_l,\bbM_m,n,\varepsilon)$ is defined analogously, with the payoff in \ref{P2} computed using $\|\cdot\|_\OP$ instead of $\|\cdot\|_\HS$.

	For fixed $m$ and $l$, as $n\to \infty$ and/or $\varepsilon\to 0$, the game becomes more difficult for the Duplicator. As in the case of graphs, the Duplicator wins $\Gamma_*(\bbM_l,\bbM_m,n,\varepsilon)$ (for $*\in \{\HS,\OP\}$) if $\bbM_l$ and~$\bbM_n$ agree (up to $\varepsilon$) on the values of continuous first-order statements (appropriately defined using one of the two norms, see~\eqref{eq.psi} below) of quantifier depth $n$. 
	
	The time has come to talk of many kinds of matrices. 
	A matrix $a$ is a \emph{projection} if it is self-adjoint ($a=a^*$) and idempotent ($a^2=a$). The \emph{rank} of a projection is the dimension of its range space $\{a\xi:\xi\}$. A matrix~$v$ is a \emph{partial isometry} if $v^*v$ is a projection. This implies that $vv^*$ is also a projection,  and ranks of the projections $v^*v$ and $vv^*$ are equal because $v$ is an isometry between their ranges. For example, $v=\begin{pmatrix} 0 & 1 \\ 0 & 0 \end{pmatrix}$ is a partial isometry in $\bbM_2$ with $vv^*=\begin{pmatrix} 1 & 0 \\ 0 & 0 \end{pmatrix}$ and $v^*v=\begin{pmatrix} 0 & 0 \\ 0 & 1 \end{pmatrix}$ projections of rank 1. Note that $v^*v+vv^*=1_2$. If $v\in \bbM_m$ is a partial isometry such that $v^*v+vv^*=1_2$, then~$m$ is exactly twice the rank of~$v^*v$, hence even. Conversely, if $m$ is even then such partial isometry exists.

	\begin{theorem}\label{T.evenodd}
		For all sufficiently small~$\varepsilon>0$  and all even 
		$m$ and  odd~$l$, Challenger wins $\Gamma_\OP(\bbM_l,\bbM_m,6,\varepsilon)$. 
	\end{theorem}

	\begin{proof}
		Since $m$ is even, some partial isometry $a_1\in \bbM_m$ satisfies $a_1^*a_1+a_1a_1^*=1_m$. 
		In the first six innings Challenger plays $a_1, a_2=a_1^*$, $a_3=a_1^*a_1$, $a_4=a_1a_1^*$,     $a_5=a_1^*a_1+a_1a_1^*$, and $a_6=1$\footnote{Clearly only $a_1$ matters. Had we defined the payoff by including $\|x^*x-(x^*x)^2)\|_\OP$ and $\|x^*x-xx^*-1\|_\OP$ among the test-formulas of \ref{P2}, this would be a `one move win' game  for the Challenger.  The payoff in an EF-game is sometimes defined by using a prescribed finite set of formulas; these games have numerous variations suitable for every occasion.} Then $a_5=1_m$ and since $l$ is odd, Duplicator cannot choose $b_1,\dots  b_5$ to satisfy the equalities in \ref{P2} exactly.  The trick is that they cannot be satisfied in $\bbM_l$ even approximately.

		If some  $b\in \bbM_m$ satisfies  $\|b^*b-(b^*b)^2\|_\OP<1/4$ and $\|bb^*+b^*b-1_m\|_\OP<1/4$, then some $v\in \bbM_m$ satisfies $v^*v=(v^*v)^2$,  and $v^*v+vv^*=1_m$. If you know how to diagonalise Hermitian matrices $v^*v$ and $vv^*$, then this is an exercise. 
		%
		Therefore winning the approximate game is as difficult for the Duplicator as winning the `on the nose' game, and they lose. 
	\end{proof}

	Just like the discrete case, the conclusion of this theorem can be reformulated as asserting that some first-order statement is evaluated differently at $\bbM_m$ and $\bbM_l$. The first-order statement associated with the proof of Theorem~\ref{T.evenodd} is the following\footnote{Evaluation of continuous first-order statements results in real numbers instead of truth values; see \cite{BYBHU}, \cite{hart2023an}.} 
	\begin{equation}\label{eq.psi}
		\psi:  \inf_{\|x\|_\OP\leq 1}\max(\|x^*x-(x^*x)^2\|_\OP,\|x^*x+xx^*-1\|_\OP). 
	\end{equation}
	The value of $\psi$ in $\bbM_n$ is 0 if $n$ is even and  $>10^{-10^{10}}$ if $n$ is odd.\footnote{There are better estimates but I am just a set theorist.}  
	A variant of the proof of Theorem~\ref{T.evenodd} relying on a modification of $\psi$  shows that for every~$k\geq 2$, there are $n$ and $\varepsilon>0$ such that if $k$ divides $m$ but it does not divide~$l$, then Challenger wins $\Gamma_\OP(\bbM_l,\bbM_m,n,\varepsilon)$.  
	
	In other words, the first-order theory of $\bbM_m$ `knows' what are the divisors of $m$. Actually, the first-order theory of $\bbM_m$ even `knows' the value of $m$, but  the point is that whether $m$ has a fixed $k$ as a divisor is determined by a single sentence, and the game associated with it has a fixed (not depending on $m$) number of innings and a fixed~$\varepsilon>0$. 
	
	How different is $\Gamma_\HS$ from $\Gamma_\OP$? Can the Challenger distinguish the parities of the dimensions in this game? Well, no. 
	
	\begin{lemma}\label{L.HS.m+1}
		For  all $n$  and $\varepsilon>0$ and all sufficiently large $m$, Duplicator wins $\Gamma_\HS(\bbM_m,\bbM_{m+1},n,\varepsilon)$. 
	\end{lemma}

	\begin{proof} Let $\alpha_m\colon \bbM_m\to \bbM_{m+1}$ be defined by  adding a zero $m+1$-st column and a zero $m+1$-st row to an $m\times m$-matrix and let $\beta_m\colon \bbM_{m+1}\to \bbM_m$ be defined by removing the $m+1$-st column and the $m+1$-st row from given $(m+1)\times (m+1)$ matrix. If $a\in \bbM_{m+1}$ and $\|a\|_\OP\leq 1$ we then have that  $\|a-\alpha(\beta(a))\|_\HS<\sqrt {2/m}$, because the sum of the squares in the entries in each row and each column is at most 1. 
		We can now describe a strategy for the Duplicator: 
		In the $j$-th inning play $\alpha(a_j)$ or $\beta(b_j)$, depending on what the Challenger played. For a large enough $m$, this assures that after $n$ moves the values in the payoff set hadn't changed by more than $\varepsilon$, resulting in a win for Duplicator.  
	\end{proof}
	
	This implies that the relations  $v^*v=(v^*v)^2$ and  $v^*v+vv^*=1$  are (even when taken together) not weakly stable with respect to $\|\cdot\|_\HS$. 
	A similar observation applies to divisibility of the dimension by any given~$k\geq 3$. What, if any, properties of large matrix algebras can be detected by their $\|\cdot\|_\HS$-first order theories? Nobody knows.

	\begin{question} \label{Q1} Given $n$ and $\varepsilon>0$, is there $\bar m$ large enough so that Duplicator wins $\Gamma_{\textrm{HS}}(\bbM_l,\bbM_m,n,\varepsilon)$ whenever $\min(m,k)\geq \bar m$? 
	\end{question}
	
	A positive answer to this question would be equivalent to asserting that  for every first-order sentence $\varphi$ the limit $\lim_{m\to \infty}\varphi^{\bbM_m}$ exists. 
	
	\section{Normalizers}
	
	\begin{quote}
		If you can't solve a problem, then there is an easier problem you can solve: find it.
		
		\hfill George P\' olya
	\end{quote} 
	Let's change the game by requiring the players to play only permutation matrices. A \emph{permutation matrix} is a matrix all of whose entries  are equal to 0 or to 1 that has exactly one 1 in each row and each column. The linear transformation associated with such matrix has the effect of permuting the basis vectors. All permutation matrices are unitary and  they are closed under products. Thus they  form a multiplicative group in~$\bbM_m$, and this group is isomorphic to $\cS_m$, the group of permutations of an $m$-element set. 
	On $\cS_m$ we have the \emph{Hamming metric} defined by  (let $\Delta(\pi,\sigma)$ be the cardinality of the set~$\{j: \pi(j)\neq \sigma(j)\}$)
	\[
	d_H(\pi,\sigma)=\frac {\Delta(\pi,\sigma)}m. 
	\]
	Let $\Gamma(\cS_m,\cS_l,n,\varepsilon)$ be defined as follows. 
	\begin{enumerate}[label = (N\arabic*), itemindent=10pt]
		\item The players alternate choosing elements of $\cS_m$ and $\cS_l$, as before.  
		\item \label{N2} Duplicator wins  if and only if 
		$|d_H(a_i a_j,a_k) -d_H(b_ib_j,b_k)|<\varepsilon$  for all $i,j$, and $k$. 
	\end{enumerate}
	This game is more similar to $\Gamma_\HS(\bbM_m,\bbM_l,n,\varepsilon)$ than it may appear. 
	First, since we are comparing groups it is only natural to ignore the addition and multiplication by scalars. 
	Second, the distance $d_H$  is similar to the metric that one gets by identifying $\pi$ and $\sigma$ with permutation matrices and computing the $\HS$-norm of the difference.\footnote{Third, from the unitary group of an operator algebra such as $\bbM_m$ one can often recover the algebra, but that is another story.}  The latter is equal to $\sqrt {2\Delta(\pi,\sigma)/n}$. However, with fixed $l$, $m$, and $n$, for all small enough $\varepsilon>0$ replacing \ref{N2} with the following would not change the outcome of the game. 
	\begin{enumerate}[label = (N\arabic*), resume, itemindent=10pt]
		\item \label{N3} Duplicator wins  if and only if
		$|\|a_i a_j-a_k\|_\HS -\|b_ib_j-b_k\|_\HS|<\varepsilon$  for all $i,j$, and $k$. 
	\end{enumerate}

	We can now answer the analog of Question~\ref{Q1}.

	\begin{theorem} There are $n$ and $\varepsilon>0$ such that there are arbitrarily large $m$ and~$l$ for which the Challenger wins $\Gamma(\cS_m,\cS_l,n,\varepsilon)$. 
	\end{theorem}

	This is  a consequence of \cite[Theorem~5.5]{alekseev2024on}, although this is not at all obvious from what we have seen so far. 
	The proof involves some fascinating ideas that we have no time to discuss such as sofic groups (\cite{Pe:Hyperlinear}),  expander graphs, Ulam-stability of approximate homomorphisms, comparison of first-order theories of alternating groups of order $p^4-1$ for primes $p\geq 7$, and something that awaits us in the final section. It is the time to  take the red pill,\footnote{After reading this far you may have an impression that it was written by a fan of the Matrix movies. This is not true at all; I just had to grab your attention somehow. Sorry.}  step out of the matrix, and take a  bird's eye view.

	\section{The really, really, really, big picture}

	I have lied to you. Question~\ref{Q1} is just a shadow of the following question asked by Sorin Popa, and negative answer to Question~\ref{Q1} implies a negative answer to Question~\ref{QII1}.\footnote{Is the converse true? Who knows. Read on.}  
	
	\begin{question}
		\label{QII1} Are all tracial ultraproducts of matrix algebras $\prod^\cU\bbM_n$ isomorphic? 
	\end{question}
	
	It seems that I have a lot of explaining to do (pun intended). 
	A nonprincipal  ultrafilter on $\bbN$ (all ultrafilters in this snapshot are nonprincipal and on $\bbN$) is a family of subsets $\cU$ of $\bbN$ that satisfies the following conditions. 
	\begin{enumerate}[label = ($\cU$\arabic*), itemindent=10pt]
		\item For all $A,B$ in $\cU$, $A\cap B\in \cU$. 
		\item For all $A\in \cU$ and $A\subseteq C\subseteq \bbN$, $C\in \cU$. 
		\item For every $A\subseteq \bbN$ exactly one of $A$ or $\bbN\setminus A$ belongs to $\cU$. 
		\item No finite set belongs to $\cU$. 
	\end{enumerate}
	The first two conditions assert that $\cU$ is a filter, while the third and fourth make it `ultra' and nonprincipal, respectively. 
	Such an object cannot be constructed without appealing to a fragment of the Axiom of Choice. While the latter has a number of fairly perverse consequences, the vast majority of contemporary mathematicians accept it as true. 
	Every ultrafilter has remarkable properties. For example, if $(x_n)$ is a bounded sequence of scalars, then there is a unique scalar~$x$ such that for every $\varepsilon>0$ the set $\{n: |x-x_n|<\varepsilon\}$ belongs to $\cU$. In this case we write $\lim_{n\to \cU} x_n=x$  and say that the sequence $(x_n)$ converges to $x$ as $m\to \cU$,  or that it converges to $x$ along $\cU$. This readily implies that for every~$m$ and $\|\cdot\|_\OP$-bounded sequence $a_n\in \bbM_m$ there is a unique $a\in \bbM_m$ such that $\lim_{n\to \cU} \|a_m-a\|_\HS=\lim_{n\to \cU} \|a_m-a\|_\OP=0$. More generally, every sequence $(x_n)$ in a compact metric (or any compact Hausdorff) space has a unique limit along $\cU$,  denoted $\lim_{n\to \cU} x_n$. 
	
	`Limits' along an ultrafilter when compactness is lost are more interesting than limits in a compact space beyond comparison.
	\subsection{Ultraproducts}
	Given graphs $G_n=(V_n,E_n)$, for $n\in \bbN$,  on $\prod_n V_n$ consider the relations
	\begin{align*}
		(a_n)\sim_{\cU} (b_n)&\Leftrightarrow \{n: a_n=b_n\}\in \cU,\\
		(a_n)\E_{\cU} (b_n)&\Leftrightarrow \{n: a_n\E_n b_n\}\in \cU.
	\end{align*}
	Then $\sim_{\cU}$ is an equivalence relation,  and  if $(a_n)\sim_{\cU} (a_n')$ and $(b_n)\sim_{\cU} (b_n')$ then $(a_n)\E_{\cU}(b_n)$  if and only $(a_n')\E_{\cU} (b_n')$. 
	The quotient structure $(\prod_n V_n/\sim_{\cU}, \E_{\cU})$ is the \emph{ultraproduct} $\prod_\cU G_n$. For every first-order sentence $\varphi$ about graphs, 
	$\varphi$ holds in $\prod_\cU G_n$ if and only if the set $\{n: \varphi$ holds in $G_n\}$ belongs to $\cU$.  This is \L o\'s's Theorem, rightfully known as the Fundamental Theorem of Ultraproducts. Theorem~\ref{T.EF-graph} then implies that the first-order theory of $\prod_\cU C_m$ does not depend on the choice of $\cU$, the~0-1 law for random graphs implies the analogous statement for (sufficiently) random graphs, and by Theorem~\ref{T.graphs-convergence} every sequence of graphs has a further subsequence with this property. 
	All of this applies to other discrete mathematical structures such as linear orderings, groups, rings\dots, but I owe you an explanation of Question~\ref{QII1}. 
	
	Back on track,  let\footnote{The reason why we consider only operator norm-bounded sequences (and for the similar requirement in \ref{P1}) is that this assures the resulting algebra is isomorphic to an algebra of bounded operators on a Hilbert space. This algebra even happens to be a II$_1$ factor.}
		$\bbM=\{(a_m): a_m\in \bbM_m\text{ and }\sup_m \|a_m\|_\OP<\infty\}$ and
	\begin{equation*}
		\textstyle c_\cU(\bbM)=\textstyle\{(a_m)\in \prod_m\bbM_m : \lim_{m\to \cU} \|a\|_\HS=0\}. 
	\end{equation*}
	Both structures are closed under all algebraic operations defined coordinatewise, and $c_\cU(\bbM)$ is an \emph{ideal} in $\bbM$: for $a\in \bbM$ and $b\in c_\cU(\bbM)$ we have $ab\in c_\cU(\bbM)$ and $ba\in c_\cU(\bbM)$. The \emph{tracial ultraproduct} of the $\bbM_m$'s is the quotient
	\[
	\textstyle\prod^\cU\bbM_m=\bbM/c_\cU(\bbM). 
	\]
	By the continuous analog of \L o\'s's Theorem,  the value of every first-order sentence in the product is the limit of its values in~$\bbM_m$ as $m\to \cU$.
	It implies that the answer to Question~\ref{Q1}  is positive if and only if the first-order theory of $\prod^\cU\bbM_m$ does not depend on the choice of $\cU$.

	By \cite{FaHaSh:Model2}, the negation of Cantor's Continuum Hypothesis (CH) implies a negative answer to Question~\ref{QII1}. Instead of offering a detour on CH this late in the paper,~I recommend that the next time you meet a set theorist  you ask whether CH is true. You may hear answers that are about as opinionated as they are different.\footnote{Also see the Paul Erd\" os anecdote in \cite[p. 251]{Fa:STCstar}.  This book  also happens to be my favourite reference for  ultraproducts, countable saturation, and whatnot \wInnocey.}
	
	A remarkable property of all ultraproducts associated with nonprincipal utrafilters on $\bbN$ is their \emph{countable saturation} (see \cite{taoultraproducts} for an enlightening exposition). All that we need to know is a consequence of this property, that if $\prod^\cU \bbM_n$ and $\prod^\cV\bbM_n$ share the first-order theory \emph{and} the CH holds, then they are isomorphic.\footnote{This is a part of the evidence that the first-order logic lies in the `Goldilocks expressiveness zone'. The nonisomorphic ultraproducts obtained when CH fails can be distinguished only by considering very large invariants coming from Saharon Shelah's non-structure theory.}
	Thus CH implies that Question~\ref{Q1} and Question~\ref{QII1} have the same answer. The answer to Question~\ref{Q1} does not depend on CH or other standard additional set-theoretic axioms; we just don't know what it is. 
	
	By replacing $\|\cdot\|_\HS$ with $\|\cdot\|_\OP$ in the definition of $c_\cU(\bbM)$, one obtains the (operator) norm ultraproduct of $\prod_\cU \bbM_m$.  Theorem~\ref{T.evenodd} and the paragraph following it imply that for every $k\geq 2$ the first-order theory of $\prod_\cU\bbM_m$ detects whether the set~$\{m: k|m\}$ belongs to~$\cU$, and the operator norm analog of Question~\ref{QII1} has a negative answer regardless of whether CH is true or not. I don't know whether the theory of $\prod_\cU\bbM_m$ can detect any other information about $\cU$. 
	
	\subsection*{Acknowledgements} I am indebted to Dima Shlyakhtenko for suggesting that I write something like this and for helpful remarks on the first draft of this paper,  to Diego Bejarano and Andrea Vaccaro for noticing a few typos, and to Vadim Alekseev for two insightful suggestions that greatly improved the exposition.

\bibliographystyle{plain}
\bibliography{games-in-the-matrix}

\begin{thebibliography}{10}

\bibitem{alekseev2024on}
V.~Alekseev and A.~Thom.
\newblock On non-isomorphic universal sofic groups.
\newblock Preprint, {arXiv}:2406.06741 [math.{GR}] (2024), 2024.

\bibitem{BYBHU}
I.~Ben~Yaacov, A.~Berenstein, C.W. Henson, and A.~Usvyatsov.
\newblock Model theory for metric structures.
\newblock In Z.~Chatzidakis et~al., editors, {\em Model Theory with
  Applications to Algebra and Analysis, Vol. II}, number 350 in London Math.
  Soc. Lecture Notes Series, pages 315--427. London Math. Soc., 2008.

\bibitem{Fa:STCstar}
I.~Farah.
\newblock {\em Combinatorial Set Theory and \cstar-algebras}.
\newblock Springer Monographs in Mathematics. Springer, second edition, 2026.

\bibitem{farah2026games}
I.~Farah.
\newblock Games in the matrix.
\newblock {\em Snapshots of modern mathematics from Oberwolfach}, to apear.

\bibitem{FaHaSh:Model2}
I.~Farah, B.~Hart, and D.~Sherman.
\newblock Model theory of operator algebras {II}: Model theory.
\newblock {\em Israel J. Math.}, 201:477--505, 2014.

\bibitem{hart2023an}
B.~Hart.
\newblock An introduction to continuous model theory.
\newblock In {\em Model theory of operator algebras}, pages 83--131. Berlin: De
  Gruyter, 2023.

\bibitem{Hodg:Model}
W.~Hodges.
\newblock {\em Model theory}, volume~42 of {\em Encyclopedia of Mathematics and
  its Applications}.
\newblock Cambridge university press, 1993.

\bibitem{Pe:Hyperlinear}
V.~Pestov.
\newblock Hyperlinear and sofic groups: a brief guide.
\newblock {\em Bull. Symbolic Logic}, 14:449--480, 2008.

\bibitem{shelah1988zero}
S.~Shelah and J.~Spencer.
\newblock Zero-one laws for sparse random graphs.
\newblock {\em J. Am. Math. Soc.}, 1(1):97--115, 1988.

\bibitem{taoultraproducts}
T.~Tao.
\newblock Ultraproducts as a bridge between hard analysis and soft analysis.
\newblock https://terrytao.wordpress.com/tag/countable-saturation/, 2011.

\end{thebibliography}

\newpage

\appendix 

\section{Graphs for a game} \label{game1}
Two graphs for a three-inning game. Can you find a winning strategy for Challenger or Duplicator?

\begin{center}
	\begin{tikzpicture}[every node/.style={circle, draw=black, fill=black, inner sep=0pt,minimum size=7pt, line width=0.8pt},thick]
		
		\def\R{3.5} 	
		
		\foreach \i in {0,...,5} {		
			\pgfmathsetmacro{\angle}{90 - \i * 60}
			\node (v\i) at (\angle:\R) {};
		}

		\foreach \i in {0,...,4} {
			\pgfmathtruncatemacro{\j}{\i+1}
			\draw (v\i)--(v\j);
		}
		\draw (v5)--(v0);
	\end{tikzpicture}
	
	\vspace{1.5cm}
	
	\begin{tikzpicture}[every node/.style={circle, draw=black, fill=black, inner sep=0pt,minimum size=7pt, line width=0.8pt},thick]
		
		\foreach \i in {0,1,2} {		
			\pgfmathsetmacro{\y}{(\i - 1) * 2.5}
			\node (l\i) at (-3, \y) {};
			\node (r\i) at ( 3, \y) {};
		}
		
		\foreach \i in {0,1,2} {		
			\foreach \j in {0,1,2} {
				\draw (l\i)--(r\j);
			}
		}
	\end{tikzpicture}
\end{center}

\newpage \section{Graphs for another game} \label{game2}
Two graphs to play games on. How many innings are necessary for a winning strategy for Challenger to exist?

\vspace{0.5cm}
\centering
\begin{tikzpicture}[every node/.style={circle, draw=black, fill=black, inner sep=0pt,minimum size=7pt, line width=0.8pt},thick]
	
	\def\R{3.5}   
	\def\N{8}     
	
	\foreach \i in {0,...,7} {
		\pgfmathsetmacro{\angle}{90 - \i * 360/\N}
		\node (v\i) at (\angle:\R) {};
	}
	
	\foreach \i in {0,...,6} {
		\pgfmathtruncatemacro{\j}{\i+1}
		\draw (v\i) -- (v\j);
	}
	\draw (v7) -- (v0);
	
	\foreach \i in {0,...,3} {
		\pgfmathtruncatemacro{\j}{\i+4}
		\draw (v\i) -- (v\j);
	}
	
\end{tikzpicture}

\vspace{1.5cm}

\begin{tikzpicture}[every node/.style={circle, draw=black, fill=black, inner sep=0pt,minimum size=7pt, line width=0.8pt},thick]
	
	\def\Ro{3.5}   
	\def\Ri{1.8}   
	
	\foreach \i in {0,...,4} {
		\pgfmathsetmacro{\angle}{90 - \i * 72}
		\node (u\i) at (\angle:\Ro) {};
	}
	
	\foreach \i in {0,...,4} {
		\pgfmathsetmacro{\angle}{90 - \i * 72}
		\node (w\i) at (\angle:\Ri) {};
	}
	
	\foreach \i in {0,...,3} {
		\pgfmathtruncatemacro{\j}{\i+1}
		\draw (u\i) -- (u\j);
	}
	\draw (u4) -- (u0);
	
	\foreach \i in {0,...,4} {
		\pgfmathtruncatemacro{\j}{mod(\i+2,5)}
		\draw (w\i) -- (w\j);
	}
	
	\foreach \i in {0,...,4} {
		\draw (u\i) -- (w\i);
	}

\end{tikzpicture}

\end{document}